\documentclass[11pt]{amsart}
\usepackage{amsfonts}
\usepackage{amsmath,amsthm,latexsym,amssymb,graphicx}
\usepackage{graphicx,color,enumerate}

\numberwithin{equation}{section}

\newtheorem{theorem}{Theorem}

\numberwithin{theorem}{section}
\numberwithin{corollary}{section}
\numberwithin{lemma}{section}
\numberwithin{definition}{section}
\numberwithin{proposition}{section}
\numberwithin{remark}{section}

\newcommand{\dint}{\ds\int}
\newcommand{\ds}{\displaystyle}

\newcommand{\R}{\mathbb R}
\newcommand{\N}{\mathbb N}

\newcommand{\medint}{-\kern  -,375cm\int}

\thanks{
$^{1}$ Dipartimento di Matematica e Applicazioni ``R. Caccioppoli'', Universit\`{a}
degli Studi di Napoli ``Federico II'', Complesso Monte S. Angelo, via Cintia
- 80126 Napoli, Italy; email: brandolini@unina.it; fchiacch@unina.it; cristina@unina.it.
\\
\indent $^2$ Institut \'Elie Cartan Nancy, Nancy Universit\'e-CNRS-INRIA, B.P. 70239, 54506 Vandoeuvre les Nancy, France; 
{antoine.henrot@univ-lorraine.fr}}

\begin{document}
\title[An optimal Poincar\'e-Wirtinger inequality in Gauss space]{An optimal Poincar\'e-Wirtinger inequality in Gauss space}
\author[\textsf{B.~Brandolini}]{Barbara Brandolini$^{1}$}
\author[\textsf{F.~Chiacchio}]{Francesco Chiacchio$^{1}$}
\author[\textsf{A.~Henrot}]{Antoine Henrot$^{2}$}
\author[\textsf{C.~Trombetti}]{Cristina Trombetti$^{1}$}

\begin{abstract}
Let $\Omega$ be a smooth, convex, unbounded domain of $\R^N$. Denote by $\mu_1(\Omega)$  the first nontrivial Neumann eigenvalue  of the Hermite operator in $\Omega$; we prove that $\mu_1(\Omega) \ge 1$. The result is sharp since equality sign is achieved when $\Omega$ is a $N$-dimensional strip. Our estimate can be equivalently viewed as an optimal Poincar\'e-Wirtinger inequality for functions belonging to the weighted Sobolev space $H^1(\Omega,d\gamma_N)$, where  $\gamma_N$ is the $N$-dimensional Gaussian measure.
\end{abstract}

\maketitle

\section{Introduction}

Let $\Omega $ be a convex domain of $\mathbb{R}^{N}$ ($N\geq 2$) and let us
denote by $d\gamma _{N}$ the standard Gaussian measure in $\mathbb{R}^{N}$,
that is%
\begin{equation*}
d\gamma _{N}=\frac{1}{(2\pi )^{N/2}}e^{-\frac{\left\vert x\right\vert ^{2}}{2%
}}dx.
\end{equation*}%
In \cite{AN} (see also \cite{Ni} and \cite{AC}) the authors prove, among
other things, that, if $\Omega$ is bounded, the first nontrivial eigenvalue $\mu _{1}(\Omega )$ of
the following problem
\begin{equation}
\left\{
\begin{array}{ll}
-\mathrm{div}\left( \exp \left( -\frac{|x|^{2}}{2}\right) Du\right) =\mu
\exp \left( -\frac{|x|^{2}}{2}\right) u & \mbox{in}\>\Omega \\
&  \\
\dfrac{\partial u}{\partial \nu }=0 & \mbox{on}\>\partial \Omega %
\end{array}%
\right.  \label{problem}
\end{equation}%
satisfies
\begin{equation}
\mu _{1}(\Omega )\geq \max \left\{ 1,\frac{1}{2}+\frac{\pi ^{2}}{\mathrm{diam%
}(\Omega )^{2}}\right\}.  \label{an1}
\end{equation}%
Here $\nu$ stands for the outward normal to $\partial\Omega$ and $\mathrm{%
diam}(\Omega )$  for the diameter of $\Omega $. It is well-known that $\mu_1(\Omega)$ can be characterized in the following variational way
$$
\mu_1(\Omega)=\min\left\{\dfrac{\dint_\Omega |D\psi|^2d\gamma_N}{\dint_\Omega \psi^2d\gamma_N}: \psi \in H^1(\Omega,d\gamma_N)\setminus\{0\},\int_\Omega \psi d\gamma_N=0\right\},
$$
where $H^1(\Omega,d\gamma_N)$ is the weighted Sobolev space defined as follows
$$
H^{1}(\Omega ,d\gamma _{N})\equiv \left\{ u\in W_{\rm loc}^{1,1}(\Omega )\text{
such that }(u,|Du|)\in L^{2}(\Omega ,d\gamma _{N})\times L^{2}(\Omega
,d\gamma _{N})\right\} ,
$$
endowed with the norm
$$
\left\Vert u\right\Vert _{H^{1}(\Omega ,d\gamma _{N})}=\left\Vert
u\right\Vert _{L^{2}(\Omega ,d\gamma _{N})}+\left\Vert Du\right\Vert
_{L^{2}(\Omega ,d\gamma _{N})}.   
$$
Incidentally note that the space $H^1(\Omega,d\gamma_N)\equiv H^1(\Omega)$ whenever $\Omega$ is a bounded domain. 

\noindent {Estimate \eqref{an1}} implies that, if $\Omega$ is a bounded, convex domain of $\R^N$, the following Poincar\'e-Wirtinger inequality holds
\begin{equation}\label{poincare}
\int_\Omega \left(u-\int_\Omega ud\gamma_N\right)^2 d\gamma_N \le \int_\Omega |Du|^2 d\gamma_N,\qquad \forall u \in H^1(\Omega,d\gamma_N).
\end{equation}
It is well-known (see for example \cite{CH}) that, when $\Omega =\R^N$, inequality \eqref{poincare} still holds true.
The purpose of this paper is to fill the gap between convex, bounded  sets and the whole $\R^N$ by proving the following sharp lower bound.

\begin{theorem}
\label{main} Let $\Omega\subset \mathbb{R}^N$ be a convex, $C^2$ domain
whose boundary satisfies a uniform interior sphere condition (see %
\eqref{uisc} below). Then
\begin{equation}  \label{mu}
\mu_1(\Omega) \ge 1,
\end{equation}
equality holding if $\Omega$ is any $N$-dimensional strip.
\end{theorem}

Our strategy consists into constructing a suitable sequence $\{\Omega _{k}\}_{k\in \N}$
of bounded convex domains invading $\Omega $ and then passing to the limit
in {\eqref{an1}}. To show that $\mu _{1}(\Omega _{k})$ converge to $\mu
_{1}(\Omega )$ one of the main ingredients is an extension theorem for
functions belonging to $H^1(\Omega_k,d\gamma_N)$ with a constant
independent of $k$ (see Theorem \ref{t-ext}, see also \cite{C}).

The structure of the differential operator in \eqref{problem} suggests the relevance of the case of unbounded sets since the density of Gaussian measure degenerates at
infinity. Moreover, in such a case, the space $H^1(\Omega,d\gamma_N)$ does no longer coincide with $H^1(\Omega)$. Indeed, when $\Omega =\mathbb{R}^{N}$, the case mostly studied by physicists, the eigenvalues of problem \eqref{problem} are the
integers and the corresponding eigenfunctions are combinations of Hermite
polynomials which clearly does not belong to $H^1(\R^N)$. 

The Hermite operator appearing in problem \eqref{problem} is widely studied in literature from many points of view. It is a classical subject in quantum mechanics (see for instance \cite{CH}) as well as in probability; indeed it is  the generator of the Ornstein-Uhlenbeck semigroup (see for example \cite{Bo}). Finally, problems of the kind \eqref{problem} are related to some functional inequalities as the well-known GrossÕs Theorem on the Sobolev Logarithmic embedding (see e. g. \cite{G, E, E2, PT, C.die, BCT2, CP}).

Note that the convexity assumption in Theorem \ref{main} cannot be relaxed;
it is enough to consider the classical example of a planar domain made by
two equal squares connected by a thin corridor. Problems linking the
geometry of a domain and the sequence of eigenvalues of a second order elliptic operator  are classical since
the estimates by Faber, Krahn or P\'olya, Sz\'eg\"o concerning the first eigenvalue of the Laplacian with Dirichlet or Neumann boundary conditions respectively. Further developments of this topic can be found for instance in \cite{AB, AB2, AFT, BNT, Ch, BNT1, BFNT}, where estimates for Dirichlet eigenvalues and eigenfunctions of linear and nonlinear operator are derived. Concerning Neumann boundary conditions we refer the reader to \cite{AB2, BCT1} and to \cite{PW, AH, BCM, ENT, FNT} for lower bounds of Neumann eigenvalues in different contexts (Laplacian, $p$-Laplacian, manifolds of constant curvature). For results in Gauss space we mention, for instance, \cite{BCF, CdB, BCT}. Clearly the above list of references is far from being exhaustive; more papers in this growing field of research are cited in \cite{K,Ke,HP,H}.


\section{Proof of Theorem \protect\ref{main}}

We recall that, given a subset $\Omega $ of $\mathbb{R}^{N}$,  $\partial \Omega $ satisfies a uniform interior sphere condition
if
\begin{equation}
{\exists\bar{r}>0}:\quad\ \forall x\in \partial \Omega \text{ \ \ \ }\exists B_{\bar{r}}\subset \Omega
\text{ \ such that \ }\overline{B_{\bar{r}}}\cap \overline{\Omega }=\left\{
x\right\} ,  \label{uisc}
\end{equation}%
where $B_{\bar{r}}$ denotes a ball with radius $\bar{r}>0.$

The proof of our result is divided in two steps. The first one provides an
extension theorem that may have an interest by its own. In the second one we
consider a sequence of convex, bounded domains $\left\{ \Omega _{k}\right\} $
invading $\Omega $ satisfying $\mu _{1}(\Omega _{k})\geq 1$ and we show that
$\lim\limits_{k}\mu _{1}(\Omega _{k})=\mu _{1}(\Omega )$. 

\begin{theorem}
\label{t-ext} Let $\Omega \subset \mathbb{R}^{N}$ be a convex, $C^{2}$
domain whose boundary satisfies \eqref{uisc} and let us denote $d_{0}=%
\mathrm{dist}(0,\partial \Omega )$. Let $u\in H^{1}(\Omega ,d\gamma _{N})$;
there exist a function $\tilde{u}\in H^{1}(\mathbb{R}^{N},d\gamma _{N})$
extending $u$ to the whole $\mathbb{R}^{N}$ and a constant $C$ such that%
\begin{equation}
||\tilde{u}||_{H^{1}(\mathbb{R}^{N},d\gamma _{N})}\leq C||u||_{H^{1}(\Omega
,d\gamma _{N})}.  \label{ext-norm}
\end{equation}%
In \eqref{ext-norm} $C=C(\bar{r},N)$ if $0\in \Omega $, while $C=C(\bar{r}%
,N,d_{0})$ if $0\notin \Omega $.
\end{theorem}

\begin{proof}[Proof of Theorem \protect\ref{t-ext}]
We distinguish two cases: $0 \in \Omega$ and $0 \notin \Omega$ {and we fix ${\tilde r}=\frac{\bar{r}}{2}$}.

Suppose first that $0 \in \Omega$. Let us denote by $d(x) = \mathrm{dist}%
(x,\partial\Omega)$ the distance of a point $x\in \mathbb{R}^N$ from $%
\partial\Omega$ and
\begin{equation*}
\Omega^{\tilde r}=\{x \in \mathbb{R}^N\setminus \overline \Omega :
d(x)<\tilde r\}, \quad \Omega_{\tilde r}=\{x \in \Omega : d(x)<\tilde r\}.
\end{equation*}
\noindent Let $u\in H^{1}(\Omega ,d\gamma _{N})$; we want to extend $u$ to $%
\mathbb{R}^{N}$ by reflection along the normal to $\partial \Omega $. Define
\begin{equation*}
\Phi :x\in \Omega _{\tilde{r}}\longrightarrow \Phi (x)=x-2d(x)Dd(x)\in
\Omega ^{\tilde{r}}.
\end{equation*}
By construction $\Phi $ is a $C^{1}$ one-to-one map; we claim that
\begin{equation}
1\leq |J_{\Phi }(x)|\leq 3^{N-1},\text{ \ \ }\forall x\in \Omega _{\tilde{r}%
}.  \label{jacobian}
\end{equation}
A straightforward computation yields
\begin{equation*}
\frac{\partial \Phi _{i}(x)}{\partial x_{j}}=\delta _{ij}-2\frac{\partial
d(x)}{\partial x_{j}}\frac{\partial d(x)}{\partial x_{i}}-2d(x) \frac{%
\partial ^{2}d(x)}{\partial x_{i}\partial x_{j}}.
\end{equation*}
By a rotation of coordinates we can assume that the $x_{N}$-axis lies in the
direction $Dd(x).$ By a further rotation of the first $N-1$ coordinates we
can also assume that the $x_{1},...,x_{N-1}$ axes lie along the principal
directions corresponding to the principal curvatures $\kappa _{1},...,\kappa
_{N-1}$ of $\partial \Omega $ at $p(x)=\frac{x+\Phi (x)}{2}.$ Clearly $p(x)$
is the projection of $x$ on $\partial \Omega $. In this coordinate system,
known as principal coordinate system at $p(x)$, it is immediate to verify
that
\begin{equation*}
\left\vert J_{\Phi }(x)\right\vert =\prod\limits_{i=1}^{N-1}\left( 1+ \frac{%
2d(x)\kappa _{i}}{1-d(x)\kappa _{i}}\right) .
\end{equation*}
Claim \eqref{jacobian} follows recalling that $d(x)< \tilde{r} =%
\frac{\bar{r}}{2}$ and $\kappa _{i}\leq \frac{1}{\overline{r}}.$

\noindent We observe that in the simplest case $N=2,$ \eqref{jacobian} has
been proven, in a different and more direct way, in \cite{BCT}.

\noindent Now define
\begin{equation*}
\overline{u}(x)=u(\Phi ^{-1}(x))\quad \forall \>x\in \Omega ^{\tilde{r}}.
\end{equation*}%
Let $\theta \in C_{0}^{\infty }(\mathbb{R}^{N})$ be a cut-off function such
that $0\leq \theta \leq 1$ in $\mathbb{R}^{N}$, $\theta =1$ on $\Omega $, $%
\theta =0$ in $\mathbb{R}^{N}\setminus (\Omega \cup\Omega ^{\tilde{r}})$ and
$|D\theta |\leq C=C(\bar{r})$. Set
\begin{equation}
\tilde{u}=\left\{
\begin{array}{ll}
u & \text{in } \>\> \Omega \\
\theta \overline{u} & \text{in} \>\> \Omega ^{\tilde{r}}\  \\
0 & \text{in} \> \>\mathbb{R}^{N}\setminus (\Omega \cup\Omega ^{\tilde{r}}).%
\end{array}%
\right.  \label{cris}
\end{equation}%
{The mediatrix of $x$ and $\Phi(x)$ is a supporting plane for $\Omega $ which is convex. Since $\Omega$} 
contains the origin, it is easy to verify that
\begin{equation}
\exp \left( \frac{-|\Phi (x)|^{2}}{2}+\frac{|x|^{2}}{2}\right) \leq 1\quad
\forall \>x\in \Omega _{\tilde{r}}.  \label{exp}
\end{equation}%
Thus, by \eqref{jacobian}, \eqref{cris} and \eqref{exp} we get
\begin{eqnarray}
\displaystyle\int_{\mathbb{R}^{N}}\tilde{u}^{2}d\gamma _{N} &=&\int_{\Omega
}u^{2}d\gamma _{N}+\int_{\Omega ^{\tilde{r}} }\tilde{u}^{2}d\gamma _{N}
 \label{u2} \\
&\leq &\int_{\Omega }u^{2}d\gamma _{N}+\int_{ \Omega _{\tilde{r}%
}}u^{2}(x)\exp \left( -\frac{|\Phi (x)|^{2}}{2}+\frac{|x|^{2}}{2}\right)
|J_{\Phi }|d\gamma _{N}  \notag \\
&\leq &C(N)\int_{\Omega }u^{2}d\gamma _{N}.  \notag
\end{eqnarray} 
On the other hand, \eqref{jacobian}, \eqref{cris}, \eqref{exp} and \eqref{u2}
imply
\begin{eqnarray*}
&\displaystyle\int_{\mathbb{R}^{N}}&|D\tilde{u}|^{2}d\gamma _{N}\leq C(N,%
\bar{r})\left[ \int_{\Omega ^{\tilde{r}}}\bar{u}^{2}d\gamma
_{N}+\int_{\Omega\cup \Omega ^{\tilde{r}}}|D\bar{u}|^{2}d\gamma _{N}\right]
\\
&\leq &C(N,\bar{r})\left[ \int_{\Omega }u^{2}d\gamma _{N}+\int_{\Omega
}|Du|^{2}d\gamma _{N}+\int_{ \Omega _{\tilde{r}}}|Du|^{2}\exp \left( -\frac{%
|\Phi (x)|^{2}}{2}+\frac{|x|^{2}}{2}\right) |J_{\Phi }|d\gamma _{N}\right]
\\
&\leq &C(N,\bar{r})\left[ \int_{\Omega }u^{2}d\gamma _{N}+\int_{\Omega
}|Du|^{2}d\gamma _{N}\right] .
\end{eqnarray*}%
Hence if $\Omega $ contains the origin (\ref{ext-norm}) holds true.

Suppose now that $0\notin \Omega $. Up to a rotation about the origin, the
translation $T:x=(x_{1},x_{2},...,x_{N})\in \mathbb{R}^{N}\rightarrow
(x_{1}-\delta ,x_{2},...,x_{N})\in \mathbb{R}^{N}$, for a fixed $\delta >d_0$,
maps $\Omega $ onto a set $T(\Omega )$ containing the origin. Define
\begin{equation*}
v(x)=v(x_{1},x_{2},...,x_{N})=u(x_{1}+\delta ,x_{2},...,x_{N})\exp \left( -%
\frac{x_{1}\delta }{2}-\frac{\delta ^{2}}{4}\right) ,\quad x\in T(\Omega );
\end{equation*}%
then
\begin{equation*}
\int_{\Omega }u^{2}d\gamma _{N}=\int_{T(\Omega )}v^{2}d\gamma _{N}.
\end{equation*}%
Since by construction $T(\Omega )$ contains the origin, there exists a
function $\tilde{v}\in H^{1}(\mathbb{R}^{N},d\gamma _{N})$ such that $\tilde{%
v}\left\vert _{T(\Omega )}\right. =v$ and
\begin{equation*}
||\tilde{v}||_{H^{1}(\mathbb{R}^{N},d\gamma _{N})}\leq C(\bar{r}%
,N)||v||_{H^{1}(T(\Omega ),d\gamma _{N})}.
\end{equation*}%
Let
\begin{equation*}
\tilde{u}(x)=\tilde{u}(x_{1},x_{2},...,x_{N})=\tilde{v}(x_{1}-\delta
,x_{2},...,x_{N})\exp \left( \frac{x_{1}\delta }{2}-\frac{\delta ^{2}}{4}%
\right) ,
\end{equation*}%
we finally get
\begin{equation*}
||\tilde{u}||_{H^{1}(\mathbb{R}^{N},d\gamma _{N})}\leq C(\bar{r}%
,N,d_{0})||u||_{H^{1}(\Omega ,d\gamma _{N})}.
\end{equation*}
\end{proof}

Using the fact that $H^{1}(\mathbb{R}^{N},d\gamma _{N})$ is compactly
embedded into $L^{2}(\mathbb{R}^{N},d\gamma _{N})$ (see for example \cite{D}) and the above extension
theorem we deduce the compact embedding of $H^{1}(\Omega ,d\gamma _{N})$
into $L^{2}(\Omega ,d\gamma _{N})$ (see also \cite{FP}). Therefore, as said in Section 1, by the
classical spectral theory on compact self-adjoint operators, $\mu
_{1}(\Omega )$ satisfies the following variational characterization
\begin{equation*}
\mu _{1}(\Omega )=\min \left\{ \frac{\dint_{\Omega }|D\psi |^{2}d\gamma _{N}}{%
\dint_{\Omega }|\psi |^{2}d\gamma _{N}}:\psi \in H^{1}(\Omega ,d\gamma
_{N})\backslash \left\{ 0\right\} ,\> \int_{\Omega }\psi d\gamma
_{N}=0\right\}.
\end{equation*}
When $\Omega $ is bounded, estimate \eqref{mu} is contained in \cite%
{AN}. Therefore, from now on $\Omega $ will denote an unbounded domain. Let $%
\Omega _{k}$ be a sequence of convex, bounded,  $C^{2}$ domains whose
boundaries satisfy \eqref{uisc} for every $k\in \mathbb{N}$, and invading $%
\Omega $ in the sense that
\begin{equation*}
\Omega _{k}\subset \Omega _{k+1}\quad \forall k\in \mathbb{N}\quad \mbox{and}%
\quad \bigcup_{k\in \mathbb{N}}\Omega _{k}=\Omega .
\end{equation*}
For the explicit construction of a sequence of this kind see for instance
\cite{BCT}.

\noindent As proven in \cite{AN} we have that
\begin{equation}\label{an}
\mu _{1}(\Omega _{k})\geq 1,\quad \forall k\in \mathbb{N},  
\end{equation} 
that can be equivalently written as
\begin{equation}\label{bic}
\int_{\Omega_k}\psi^2d\gamma_N \le \int_{\Omega_k}|D\psi|^2d\gamma_N, \qquad \forall \psi \in H^1(\Omega_k,d\gamma_N): \int_{\Omega_k}\psi d\gamma_N=0.
\end{equation}
Now we want to pass to the limit in \eqref{an}. To this aim consider the
operator
\begin{equation*}
A_{k}:f\in L^{2}(\Omega ,d\gamma _{N}):\int_{\Omega }fd\gamma
_{N}=0\longrightarrow \tilde{u}_{k}\in H^{1}(\Omega ,d\gamma _{N}),
\end{equation*}%
where $\tilde{u}_{k}$ is the extension provided in Theorem \ref{t-ext} of
the solution $u_{k}\in H^1(\Omega_k,d\gamma_N)$ to the following problem
\begin{equation}
\left\{
\begin{array}{ll}
-\mathrm{div}\left( \exp \left( -\frac{|x|^{2}}{2}\right) Du_{k}\right)
=(f-c_{k})\exp \left( -\frac{|x|^{2}}{2}\right)  & \mbox{in}\>\Omega _{k} \\
&  \\
\dfrac{\partial u_{k}}{\partial \nu }=0 & \mbox{on}\>\partial \Omega _{k}\\
&\\
\dint_{\Omega_k}u_kd\gamma_N=0,%
\end{array}%
\right.   \label{prob_k}
\end{equation}%
where $c_{k}=:\int_{\Omega _{k}}fd\gamma _{N}$ and $\nu$ is the outward normal to $\partial\Omega_k$. Observe that Lax-Milgram theorem ensures the existence and uniqueness of $u_k$. Moreover $\gamma _{N}(\Omega
\setminus \Omega _{k})\rightarrow 0$ implies $c_{k}\rightarrow 0$. We also introduce the operator
\begin{equation*}
A:f\in L^{2}(\Omega ,d\gamma _{N}):\int_{\Omega }fd\gamma
_{N}=0\longrightarrow u\in H^{1}(\Omega ,d\gamma _{N}),
\end{equation*}%
where $u$ is the unique solution, whose existence is guaranteed by Lax-Milgram theorem, to the problem
\begin{equation}
\left\{
\begin{array}{ll}
-\mathrm{div}\left( \exp \left( -\frac{|x|^{2}}{2}\right) Du\right) =f\exp
\left( -\frac{|x|^{2}}{2}\right)  & \mbox{in}\>\Omega  \\
&  \\
\dfrac{\partial u}{\partial \nu }=0 & \mbox{on}\>\partial \Omega \\
&\\
\dint_{\Omega }u d\gamma_N=0.%
\end{array}%
\right.   \label{bi}
\end{equation}%
Using $u_{k}$ as test function in \eqref{prob_k}, from Schwarz inequality we deduce
\begin{equation*}
\int_{\Omega _{k}}|Du_{k}|^{2}d\gamma _{N}=\int_{\Omega
_{k}}u_{k}(f-c_{k})d\gamma _{N}\leq \left( \int_{\Omega
_{k}} u_{k}^{2}d\gamma _{n}\right) ^{1/2}\left( \int_{\Omega
_{k}}(f-c_{k})^{2}d\gamma _{N}\right) ^{1/2}.   
\end{equation*}%
Using \eqref{bic} and recalling that $c_k \to 0$, we get 
$$
\int_{\Omega_k} |Du_k|^2d\gamma_N \le C_1 \left(\int_\Omega f^2d\gamma_N\right)^{1/2}+C_2,
$$
where $C_1$, $C_2$ are positive constants whose values are independent of $k$. The above inequality together with \eqref{bic} yield 
$$
\int_{\Omega_k}u_k^2d\gamma_N+\int_{\Omega_k}|Du_k|^2 d\gamma_N \le C,
$$
where $C$ is a positive constant whose value is independent of $k$. From \eqref{ext-norm} we deduce
that  the sequence $ 
\left\{ \tilde{u}_{k}\right\}_{k \in \N}$ is bounded in $H^{1}(\Omega ,d\gamma _{N}).$
Since the embedding of $H^{1}(\Omega ,d\gamma _{N})$ into $ 
L^{2}(\Omega ,d\gamma _{N})$ is compact, there exists a
(not relabeled) subsequence $\tilde{u}_{k}$ such that $\tilde{u}%
_{k}\rightharpoonup v$ in $H^{1}(\Omega ,d\gamma _{N})$, $\tilde{u}%
_{k}\rightarrow v$ in $L^{2}(\Omega ,d\gamma _{N})$ and a.e. in $\Omega $.
In fact $v$ coincides with $u$ since they both solve the same problem \eqref{bi}.
Indeed, let $\phi \in C^{\infty }(\Omega )$.  Recalling that $\gamma
_{N}(\Omega \setminus \Omega _{k})\to 0$ and $\Omega_k \subset \Omega$, we get
\begin{eqnarray*}
\int_{\Omega }DvD\phi d\gamma _{N} &=&\lim_{k}\int_{\Omega }D\tilde{u}%
_{k}D\phi d\gamma _{N}=\lim_{k}\left( \int_{\Omega _{k}}Du_{k}D\phi d\gamma
_{N}+\int_{\Omega \setminus \Omega _{k}}D\tilde{u}_{k}D\phi d\gamma
_{N}\right)  \\
&=&\lim_{k}\int_{\Omega _{k}}(f-c_{k})\phi d\gamma _{N}=\int_{\Omega }f\phi
d\gamma _{N}.
\end{eqnarray*}%
Finally, as $k$ goes to $+\infty $,
\begin{equation}
||(A_{k}-A)f||_{L^{2}(\Omega ,d\gamma _{N})}=||\tilde{u}_{k}-u||_{L^{2}(%
\Omega ,d\gamma _{N})}\rightarrow 0.  \label{Ak_A}
\end{equation}%
The compact embedding of $H^1(\Omega,d\gamma_N)$ into $L^2(\Omega,d\gamma_N)$ and (\ref{Ak_A}) allow us to adapt
Theorems 2.3.1 and 2.3.2 in \cite{H} to conclude that the operators $A_k$ strongly converge to $A$ and hence
\begin{equation*}
\mu _{1}(\Omega _{k})\rightarrow \mu _{1}(\Omega ).
\end{equation*}%

Finally we prove the optimality of our estimate \eqref{mu}. Consider the $N$-dimensional strip
$$
S_a=\{x=(x_1,x_2,...,x_N)\in \R^N: \> -a<x_1<a,\> x_2,...,x_n\in\R\},\quad a \in (0,+\infty).
$$
The eigenfunctions are factorized and can be written as linear combinations of products of homogeneous Hermite polynomials $H_{n_1}(x_1)$, $H_{n_2}(x_2)$,...,$H_{n_N}(x_N)$. We recall that the Hermite polynomials in one variable are defined by
$$
H_n(t)=(-1)^ne^{t^2/2}\frac{d^n}{dt^n}e^{-t^2/2},\quad n \in \N \cup \{0\},
$$
and they constitute a complete set of eigenfunctions to problem \eqref{problem} when $\Omega=\R$; more precisely
$$
-\left(e^{-t^2/2}H_n'(t)\right)'=ne^{-t^2/2}H_n(t),\quad n \in \N \cup \{0\}.
$$
Denote by $\lambda_1(-a,a)$ the first Dirichlet eigenvalue of the one-dimensional Hermite operator in the interval $(-a,a)$. One can easily verify that
$$
\mu_1(-a,a)=\lambda_1(-a,a)+1>1=\mu_1(\R).
$$
Therefore  $\mu_1(S_a)=1$ for every $a \in (0,+\infty)$ and a corresponding eigenfunction is, for instance, 
{$H_1(x_2)=x_2$}.

\end{document}